\documentclass[a4paper,11pt]{article}

\usepackage{hyperref,proof}

\usepackage{latexsym}
\usepackage{amsmath, amsthm}
\usepackage{amscd}
\usepackage{amssymb,enumerate}

\usepackage{amsthm}

\def\RCAo{\mathsf{RCA_0}}

\def\RCA{\mathsf{RCA_0}}

\def\ACA{\mathsf{ACA_0}}

\def\N{\mathbb{N}}

\def\P2{\Pi^1_2}

\def\RT{\mathrm{RT}}

\newcounter{menum}
{\begin{enumerate}%
\setcounter{enumi}{#1}}%
{\setcounter{menum}{\value{enumi}}\end{enumerate}}

\newtheorem{thm}{Theorem}[section]
\newtheorem{theorem}[thm]{Theorem}

\newtheorem{claim}{Claim}[thm]
\newtheorem*{claim*}{Claim}

\newtheorem{proposition}[thm]{Proposition}
\newtheorem{lemma}[thm]{Lemma}

\theoremstyle{definition}
\newtheorem{defi}{Definition}[section]
\newtheorem{definition}[defi]{Definition}

\newtheorem{remark}[thm]{Remark}

\newtheorem{example}[defi]{Example}

\usepackage{url}

\usepackage{tikz}

\def\RCA{\mathsf{RCA_0}}
\newcommand{\is}[1]{\mathsf{I}\Sigma^0_{#1}}
\newcommand{\bs}[1]{\mathsf{B}\Sigma^0_{#1}}

\newcommand{\clG} {{\sf cl}(\mathcal{G})}
\def\RT{\mathrm{RT}}
\newcommand{\arcd}[2] {#1\xrightarrow{\downarrow} #2}
\newcommand{\arce}[2] {#1\xrightarrow{\Downarrow} #2}
\newcommand{\mboxb}[1]{\mbox{ \textbf{#1} }}
\newcommand{\mbb}[1]{\mbox{ \textbf{#1} }}
\newcommand{\mb}{\mathbf}
\newcommand{\mcl}{\mathcal}

\newcommand{\last}{\mathrm{last}}
\newcommand\STAR{\triangle\text{-}\mathrm{RT}}
\newcommand{\SPP}{\mathrm{sRT}^1}
\newcommand\NN{\mathbb{N}}
\def\RCAo{\mathsf{RCA_0}}
\newcommand{\var}{\mathrm{Par}}
\newcommand\fun{\mathrm{Fun}}
\newcommand\arity{\mathrm{ar}}
\newcommand\aexp{\mathrm{AExp}}
\newcommand\bexp{\mathrm{BExp}}
\newcommand\eexp{\mathrm{Exp}}
\newcommand\ddef{\mathrm{Def}}
\newcommand\prog{\mathrm{Prog}}
\newcommand\op{\mathrm{Op}}

\newcommand{\bp}[1]{\left\lbrace #1 \right\rbrace}
\def\ACA{\mathsf{ACA_0}}

\title{The strength of the SCT criterion}

\author{Emanuele Frittation
	\and Silvia Steila \and Keita Yokoyama}
\date{}

\begin{document}

\maketitle

\begin{abstract}
We undertake the study of size-change analysis in the context of Reverse Mathematics. In particular, we prove that the SCT criterion \cite[Theorem 4]{Jones1} is equivalent to $\is2$ over $\RCA$. 
\end{abstract}


\section{Introduction}
Ramsey's theorem for pairs ($\RT^2$) is one of the main characters in Reverse Mathematics. It states that for any natural number $k$ and for any edge coloring of the complete graph with countably many nodes in $k$-many colors, there exists an infinite homogeneous set, i.e. there exists an infinite subset of  nodes whose any two elements are connected with the same color \cite{Ramsey}.

As highlighted by Gasarch \cite{Gasarch}, Ramsey's theorem for pairs can be used to prove termination. For instance, Podelski and Rybalchenko characterized the termination of transition based programs as a property of well-founded relations by using Ramsey's theorem for pairs \cite{Podelski}. In \cite{RMBound} we started investigating the termination analysis from the point of view of Reverse Mathematics.  We proved the equivalence between the termination theorem of Podelski and Rybalchenko and a corollary of Ramsey's theorem for pairs, which is weaker than Ramsey's theorem for pairs itself. 

The termination theorem is not the only result which characterizes the termination of some class of programs. 
In \cite{Jones1} Lee, Jones and Ben-Amram introduced the notion of size-change termination (SCT) for first order functional programs. Size-change analysis is a general method for \emph{automated termination proofs}. In fact, this method has been applied in the termination analysis of  higher-order programs \cite{JonesB04}, logic programs \cite{Codish}, and  term rewrite systems \cite{Thiemann}. 

Informally, a program is size-change terminating (SCT) if every infinite state transition sequence
would cause an infinite sequence of data values  which is weakly decreasing and strictly decreasing infinitely many times. If the domain of data values is well-founded, such as the natural numbers, there cannot be such a sequence, thus SCT is a sufficient condition for termination  \cite[Theorem 1]{Jones1}.  

Size-change termination is based on the notion of size-change graph (see Subsection \ref{graph}). 
Given a first order functional program $P$ we associate to every call $f \to g$ a bipartite graph which describes the relation between source and target parameter values. These graphs are called size-change graphs. 

In this paper we start the investigation of size-change termination in the framework of Reverse Mathematics. In particular, we analyse the following criterion  for testing SCT \cite[Theorem 4]{Jones1}:

\begin{theorem}[SCT criterion]\label{sct criterion}
	Let $\mathcal{G}$ be a set of size-change graphs for a first order functional program $P$. Then $\mcl G$ is SCT iff every idempotent $G\in{\sf cl}(\mathcal{G})$ has an arc $\arcd{x}{x}$. 
\end{theorem}

The original proof of the SCT criterion is based on Ramsey's theorem for pairs. In this paper we show that this is far from optimal and pinpoint the exact strength of the SCT criterion from the point of view of Reverse Mathematics. For our analysis we consider the following version, where we consider size-change graphs only. 

\begin{theorem}[SCT criterion for graphs]
	Let $\mathcal{G}$ be a set of size-change graphs.  Then $\mathcal{G}$ is SCT iff every idempotent $G\in{\sf cl}(\mathcal{G})$ has an arc $\arcd{x}{x}$. 
\end{theorem}

To the aim of studying the strength of the SCT criterion we introduce and study a corollary of Ramsey's theorem for pairs, called Triangle Ramsey's theorem ($\STAR$).  It states that for any natural number $k$ and for any edge coloring of the complete graph with countably many nodes in $k$-many colors, there is some node which is, for some color $i \in k$, the first node of infinitely many triangles homogeneous in color $i$. As far as we know this corollary does not appear in the literature. 

We show that $\STAR$ implies the SCT criterion and that the SCT criterion implies the Strong Pigeonhole Principle $(\SPP)$. From these (and some further) results we are able to conclude that both  SCT criterion  and $\STAR$ are equivalent to $\Sigma^0_2$-induction ($\is2$).

\begin{theorem}[$\RCA$] The following are equivalent:
	\begin{enumerate}
		\item $\is2$
		\item $\STAR$
		\item SCT criterion
	\end{enumerate}
\end{theorem}

\subsection{Notation}    

Given a set $X\subseteq\N$, let  $[X]^2$ denote the set of $2$-element subsets of $X$. As usual, we identify $[X]^2$ with the set $\{(x,y)\colon x,y \in X \land x<y\}$. 
We also identify a natural number $k$  with the set $\{0,\ \dots,\ k-1\}$.  For $k \in \NN$,  we call a function  $c:[\NN]^2 \to k$ a \emph{coloring} of $[\NN]^2$ in $k$-many colors. 

For a set $X\subseteq\N$, 
$X^{<\NN}$ denotes the set of finite sequences of elements in $X$. 
Given a set $X$ and a sequence $\sigma \in X^{<\NN}$  we denote by $|\sigma|$ the length of the sequence, by $\last(\sigma)$ the last element of the sequence and by $\sigma(i)$ the $i$-th element of the sequence,  if it exists. Note that $k^{< \NN}$ is the set of finite sequences of natural numbers less than $k$.

\subsection{Reverse Mathematics}

Reverse Mathematics is a program in mathematical logic introduced by Harvey Friedman in \cite{Friedman75}, which stems from the following question. Given a theorem of ordinary mathematics, what is the weakest subsystem of second order arithmetic in which it is provable?

Amongst the several subsystems of second order arithmetic (see \cite{SOSOA} for a detailed description), in this paper we consider only few  extensions of $\RCA$ (Recursive Comprehension Axiom). $\RCA$ is the standard base system of Reverse Mathematics.
It consists of the usual axioms of first order arithmetic for $0,1,+,\times,<$, induction for $\Sigma^0_1$-formulas ($\is1$) and comprehension for $\Delta^0_1$-formulas.  

The infinite pigeonhole principle ($\RT^1$) and Ramsey's theorem for pairs ($\RT^2$) are defined as follows. 
\begin{description}
	\item[($\RT^1_k$)] For any $c\colon\N\to k$ there exists $i<k$ such that $c(x)=i$ for infinitely many $x$.
	\item[($\RT^1$)] $\forall k\in\N\ \RT^1_k$. \smallskip
	\item[($\RT^2_k$)] For any $ c: [\NN]^2 \to k$ there exists an infinite homogeneous set $X \subseteq \NN$, that is 
	$c\restriction [X]^2$ is constant.
	\item[($\RT^2$)] $\forall k\in\N\ \RT^2_k$.
\end{description}

Let $\is2$ be induction for $\Sigma^0_2$-formulas. It is known that  $\RT^2$ implies the bounding principle for $\Sigma^0_3$-formulas ($\bs3$) over $\RCA$  \cite{Hirst87}, and so in particular $\is2$. As a side result here we provide a different proof of the fact that $\RT^2$ implies $\is2$. Indeed we introduce an immediate consequence of $\RT^2$, the Triangle Ramsey's theorem ($\STAR$), which turns out to be equivalent to $\is2$.


\begin{description}
	\item[($\STAR_k$)] For any coloring $c: [\NN]^2 \to k$ there exist $i\in k$ and $t\in\N$ such that $c(t, m) = c(t, l) = c(m, l) = i$ for infinitely many pairs $m<l$.
	\item[($\STAR$)]$\forall k\in\N \ \STAR_k$
\end{description}

\section{The SCT framework}

In this section we describe the size-change method for first order functional programs as in \cite{Jones1}. All the definitions are made in $\RCAo$ except for the semantic notion of \emph{safety}. 

\subsection{Syntax}

We consider the following basic first order functional language:

\begin{align*}
x\in \var  & \  \text{ parameter identifier}\\
f\in \fun   & \ \text{ function identifier}\\
o\in \op   & \ \text{ primitive operator}\\
a\in \aexp  & \ \text{ arithmetic expression}  \\
&::= \ x \mid x+1 \mid x-1 \mid o(a,\ldots,a) \mid f(a,\ldots,a) \\
b \in \bexp   & \ \text{ boolean expression}       \\
&::= \  x=0 \mid x=1 \mid x<y \mid x\leq y \mid b\land b \mid b\lor b \mid \neg b \\
e \in \eexp & \ \text{ expression} \\ 
&::= \  a \mid \mboxb{if } b \mboxb{ then } e \mboxb{ else } e \\
d \in \ddef &  \ \text{ function definition} \\
&::=\ f(x_0,\ldots,x_{n-1})= e \\
P\in \prog & \ \text{ program} \\
&::= \ d_0,\ldots,d_{m-1} 
\end{align*}

\begin{remark}
	This language is Turing complete.
\end{remark}

A program $P$ is a list of finitely many  defining equations  $f(x_0,\ldots,x_{n-1})=e^f$, where $f\in\fun$ and $e^f$ is an expression, called the \emph{body} of $f$.   Let  $x_0,\ldots,x_{n-1}$ be the \emph{parameters} of $f$, denoted $\var(f)$, and let $n$ be the \emph{arity} of $f$, denoted  $\arity(f)$. 

By $\fun(P)$ we denote the set of functions of $P$. 
We also assume that a program $P$ specifies an \emph{initial} function $f\in\fun(P)$. The idea is that  $P$ computes the (partial) function $f:\N^{\arity(f)}\to\N$.

In \cite{Jones1} the expression evaluation is based on a \emph{left-to-right call-by-value} strategy  given by \emph{denotational semantics}. $\RCA$ is not capable to formalize denotational semantics, and hence we need to consider other approaches if we want to study termination over $\RCA$ (for instance, by \emph{operational semantics}).  Anyway we do not formally discuss semantics.  For the sake of exposition, it is enough to say that one evaluates a program function $f$ given an assignment of values $\mb u$  to its parameters (i.e. an element of $\NN^{\arity(f)}$) by evaluating the body of $f$, that is $f(\mb u)=e^f(\mb u)$.


\begin{example}[P\'{e}ter-Ackermann]
	
	\[
	\begin{split}
	A(x,y) = &\mbb{ if } x=0 \mbb{ then } y+1 \mbb{ else}\\
	&\mbb{ if } y=0 \mbb{ then }  A(x-1,1)\\
	& \mbb{ else }  A(x-1,A(x,y-1))
	\end{split}
	\]
\end{example}

\subsection{Size-change graphs}\label{graph}
In order to express the notion of size-change termination, first of all we need the definition of size-change graph (see \cite[Definition 3]{Jones1}).

\begin{definition}[size-change graph]
	Let $P$ be a program and $f,\ g\in\fun(P)$. A \emph{size-change graph} $G:f \to g$ for $P$ is a bipartite  graph  on $(\var(f),\var(g))$. The set of edges is a subset of $\var(f) \times \bp{\downarrow, \Downarrow} \times \var(g)$ such that there is at most one edge for any $x\in \var(f)$ and $y \in \var(g)$.
\end{definition}

We say that $f$ is the \emph{source} function of $G$ and $g$ is the \emph{target} function of $G$.
We call  $(x,\downarrow,y)$ the \emph{decreasing edge} (strict arc), and we denote it by $\arcd{x}{y}$. We call $(x,\Downarrow,y)$ the \emph{weakly-decreasing edge} (non-strict arc), and we denote it by $\arce{x}{y}$. We write $x\to y\in G$ as a shorthand for $\arcd{x}{y}\in G\lor \arce{x}{y}\in G$.

Note that the absence of edges between two parameters $x$ and $y$ in the size-change graph $G$  indicates either an unknown or an increasing relation in the call $f\to g$.

Informally, a size-change graph is an approximation of the \emph{state transition} relation induced by the program. A size-change graph $G:f\to g$ for a call  $\tau: f\to g$ is \emph{safe} if it reflects the relationship between the parameter values in the program call.
%

In more detail, a \emph{state} of a program $P$ is a pair $(f,\mb u)$,  where $f\in\fun(P)$ and $\mb u$ is a tuple of length  $\arity(f)$. If in the body of $f\in\fun(P)$ there is a call
\[
\dots \tau: g(e_0, \dots, e_{m-1})
\]
we define a \emph{state transition} $(f,\mb u)\xrightarrow{\tau}(g,\mb v)$ 
to be a pair of states such that $\mb v$ is the sequence of values obtained by the expressions $(e_0,\dots, e_{m-1})$ when $f$ is evaluated with values $\mb u$. 

Let $\var(f)=\{x_0,\ldots, x_{n-1}\}$ and $\var(g)=\{y_0,\ldots y_{m-1}\}$. We say that   a size-change graph $G: f\to g$ is \emph{safe}  for $\tau$ if every edge is safe, where an edge $x_i \xrightarrow{r} y_j$ is safe 
if for any $\mb u\in \NN^{n}$ and $\mb v \in \NN^{m}$ such that $(f,\mb u)\xrightarrow{\tau}(g,\mb v)$, $r=\mathord\downarrow$ implies that $\mb u_i>\mb v_j$ and $r=\mathord\Downarrow$ implies that $ \mb u_i\geq \mb v_j$.

Note for instance that the size-change graph without edges is always safe.

\begin{example}[P\'{e}ter-Ackermann] \label{exampleAck}~
	\begin{center}
		\begin{minipage}[c]{0.3\textwidth}
			\[\begin{tikzpicture}[xscale=1,yscale=1]
			
			\node (x1) at (0, 1) {$x$};
			\node (y1) at (0, 0) {$y$};
			\node (x2) at (2, 1) {$x$};
			\node (y2) at (2, 0) {$y$};
			\path (x1) edge [->]node [auto] {${\downarrow} $} (x2);
			\node (G) at (1,-1) {$G_{0,1}:A\to A$};
			\end{tikzpicture}
			\]
			
		\end{minipage}
		\begin{minipage}[c]{0.3\textwidth}
			\[\begin{tikzpicture}[xscale=1,yscale=1]
			
			\node (x1) at (0, 1) {$x$};
			\node (y1) at (0, 0) {$y$};
			\node (x2) at (2, 1) {$x$};
			\node (y2) at (2, 0) {$y$};
			\path (x1) edge [->]node [auto] {${\Downarrow} $} (x2);
			\path (y1) edge [->]node [auto] {${\downarrow} $} (y2);
			\node (H) at (1,-1) {$G_2:A\to A$};
			\end{tikzpicture}
			\]
		\end{minipage}
	\end{center}
	
	The size-change graph $G_{0,1}$ safely describes both calls  $\tau_0 : A(x-1,1)$ and $\tau_1: A(x-1,A(x,y-1))$. The size-change graph $G_2$ safely describes call $\tau_2: A(x,y-1)$. 
\end{example}

Note that we could have assumed that for any parameter in the target there is at most one edge, since in every call of the programs we consider any parameter value in the target depends at most from one parameter in the source.   However this restriction is not essential. Note also that the SCT framework has been generalized in order to deal with other kinds of \emph{monotonicity constraints} \cite{BenAmram}, where SCT only deals with two constraints $x>y$ (a strict arc) and $x\geq y$ (a non-strict arc).

Nonetheless we want to emphasize that the notion of size-change graph is clearly independent of that of a program and so we can define it directly. For simplicity we may assume that every function $f\in\fun$ comes with a set of parameters $\var(f)$ of size $\arity(f)$. 

\begin{definition}[size-change graph]
	Let $f,g\in\fun$. A \emph{size-change graph} $G:f \to g$ is a bipartite  graph on  $(\var(f),\var(g))$. The set of edges is a subset of $\var(f) \times \bp{\downarrow, \Downarrow} \times \var(g)$ such there is at most one edge for any $x\in \var(f)$ and $y \in \var(g)$.
\end{definition}

\subsection{SCT criterion}

\begin{definition}[composition]
	As in \cite{Jones}, given two size-change graphs $G_0:f\to g$ and $G_1:g\to h$ we define their \emph{composition} $G_0;G_1:f\to h$. The composition of two edges $\arce{x}{y}$ and $\arce{y}{z}$ is one edge $\arce{x}{z}$. In all other cases the composition of two edges from $x$ to $y$ and from $y$ to $z$ is the edge $\arcd{x}{z}$. Formally, $G_0;G_1$ is the size-change graph with the following set of edges:
	\[
	\begin{split}
	E= \{\arcd{x}{z} : \ &\exists y \in \var(g)\  \exists r\in \bp{\downarrow,\Downarrow}((\arcd{x}{y}\in G_0 \wedge y\xrightarrow{r}z \in G_1) \\
	& \vee (x\xrightarrow{r}y\in G_0  \wedge \arcd{y}{z}\in G_1))\}\\
	\cup \{\arce{x}{z} : \ &\exists y \in \var(g)(\arce{x}{y}\in G_0 \wedge \arce{y}{z}\in G_1) \wedge \forall y \in \var(g) \\
	&\forall r, r'\in \bp{\downarrow,\Downarrow} ((x \xrightarrow{r} y \in G_0 \wedge y \xrightarrow{r'} z \in G_1) \implies r=r'=\mathord \Downarrow)\}.
	\end{split}
	\]
\end{definition}

Observe that the composition operator ``{$;$}'' is associative. Moreover we say that the size-change graph $G$ is \emph{idempotent} if $G;G=G$.

Given a finite set of size-change graphs $\mathcal{G}$, $\clG$ is the smallest set which contains $\mathcal{G}$ and is closed by composition. Formally $\clG$ is the smallest set such that 
\begin{itemize}
	\item $\mathcal{G}\subseteq \clG$;
	\item If $G_0:f \to g$ and $G_1:g \to h$ are in $\clG$, then $G_0;G_1\in \clG$.
\end{itemize}


\begin{definition}[multipath]
	A \emph{multipath} $\mathcal{M}$ is a graph sequence $G_0,\dots, G_n,\dots$ such that the target function of $G_i$ is the source function of $G_{i+1}$ for all $i$. A \emph{thread} is a connected path of edges in $\mathcal{M}$ that starts at some $G_t$, where $t \in \NN$. A multipath $\mathcal{M}$ has \emph{infinite descent} if some thread in $\mathcal{M}$ contains infinitely many decreasing edges. 
\end{definition}

\begin{definition}[description]
	A \emph{description} $\mathcal{G}$ of $P$ is a finite set of size-change graphs such that to every  call $\tau:f\to g$  of $P$ corresponds  exactly one  $G_\tau\in\mcl G$. 
\end{definition}

A description $\mathcal{G}$ of $P$ is \emph{safe} if each graph in $\mathcal{G}$ is safe. 
Note that there are finitely many descriptions, and in particular finitely many safe descriptions.

\begin{definition}[SCT description]
	Let $\mathcal{G}$ be a description of $P$. We say that $\mathcal{G}$ is \emph{size-change terminating} (SCT)  if every infinite multipath $\mcl M=G_0,\dots, G_n,\dots$, where every graph $G_n\in\mcl G$, has an infinite descent.
\end{definition}

It is clear that a program $P$ with a safe SCT description does not have infinite state transition sequences. 
Thus the existence of a safe SCT description is a sufficient condition for termination.


We now can state the SCT criterion.
\begin{theorem}[SCT criterion]
	Let $\mathcal{G}$ be a description of $P$. Then $\mathcal{G}$ is SCT  iff every idempotent $G\in{\sf cl}(\mathcal{G})$ has an arc $\arcd{x}{x}$.
\end{theorem}

To the aim of analysing in Reverse Mathematics it is convenient to state the SCT criterion for arbitrary sets of size-change graphs.

\begin{definition}[SCT criterion for graphs]
	Let $\mathcal{G}$ be a finite set of size-change graphs.  Then $\mathcal{G}$ is SCT iff every idempotent $G\in{\sf cl}(\mathcal{G})$ has an arc $\arcd{x}{x}$. 
\end{definition}

It is not difficult to see that the two formulations of the SCT criterion are equivalent. In fact, given a finite set $\mathcal{G}$ of size-change graphs, it is straightforward to define a program $P$ such that $\mathcal{G}$ is a description of $P$. 
In more detail,  let $f_0, \dots, f_m$ be the finite set of source and target functions of $\mcl G$. Without loss of generality, we may assume that all functions have the same arity $n \in \NN$.  For any $i$, let $f_{i_0}, \dots, f_{i_{k-1}}$ be the functions (with repetition if there are more graphs with the same source and target functions) which correspond to the target of a graph whose source is $f_i$. Write the code:
\begin{align*}
f_i(x_0, \dots x_{n-1}) &= \tau_0:  f_{i_0}( e_0^0, \dots, e_{n-1}^0) &\mbox{if }  x_0= 0.\\
&=  \dots \\
&= \tau_{k-1}:  f_{i_{k-1}}( e_0^{k-1}, \dots, e_{n-1}^{k-1}) &\mbox{if }  x_0= k-1.\\
\end{align*}
where the expression $e_j^h$ is determined by the source and the kind of the edge to $x_j$ in the corresponding graph, if such an edge exists. Otherwise it is $x_j+1$. 

The union of these codes is a  program $P_\mathcal{G}$. Of course,  $\mathcal{G}$ is a description of $P_\mathcal{G}$. Therefore:

\begin{proposition}[$\RCAo$]
	The following are equivalent:
	\begin{enumerate}
		\item SCT criterion
		\item SCT criterion for graphs
	\end{enumerate}
\end{proposition}

\section{Proving the SCT criterion}

The classical proof of the SCT criterion \cite{Jones} uses Ramsey's theorem for pairs. Actually,
what we really need is that there exist infinitely many monochromatic triangles which share a fixed vertex: we need the homogeneous cliques in order to prove that the graph is idempotent and that there are infinitely many strictly decreasing edges in the thread and we need  that they share a fixed vertex in order to guarantee the continuity of the path. This is why we introduce the principle $\STAR$. 

\begin{description}
	\item[($\STAR_k$)] For any coloring $c: [\NN]^2 \to k$ there exist $i\in k$ and $t\in\N$ such that $c(t, m) = c(t, l) = c(m, l) = i$ for infinitely many pairs $\{m,l\}$.
	\item[($\STAR$)]$\forall k\in\N \ \STAR_k$.
\end{description}

We also introduce the following strengthening of the infinite pigeonhole principle:
\begin{description}
	\item[($\SPP_k$)] For any coloring $c\colon\N\to k$ there exists $I\subseteq k$ such that $i\in I$ iff $i<k$ and 
	$c(x)=i$ for infinitely many $x$.
	\item[($\SPP$)] $\forall k\in\N\ \SPP_k$.
\end{description}

For the reversal we use the fact that $\SPP$ is equivalent to $\Sigma^0_2$-induction. 

\begin{lemma}[$\RCA$]
	The following are equivalent:
	\begin{enumerate}
		\item $\is2$
		\item $\SPP$
	\end{enumerate}
\end{lemma}
\begin{proof}
	It is well-known that $\is2$ is equivalent over $\RCA$ to \emph{bounded comprehension} for $\Pi^0_2$-formulas, that is the axiom schema
	\[ \forall k\ \exists X\ \forall i\ (i \in X\leftrightarrow i<k\land \varphi(i)), \]
	where  $\varphi$ is $\Pi^0_2$. It immediately follows that $\is2$ implies $\SPP$. Let us show that $\SPP$ implies bounded $\Pi^0_2$-comprehension. Let $\varphi(i)=\forall x\exists y\ \theta(i,x,y)$. 
	We define $c\colon\N\to k+1$ by primitive recursion as follows: 
	\begin{enumerate}
		\item Let $s=0$ and $x_i=0$ for all $i<k$;
		\item Suppose we have defined $c(x)$ for every $x<s$. For all $i<k$, if $\exists y<s\ \theta(i,x_i,y)$, let $c(s+i)=i$ and $x_i=x_i+1$. Otherwise let $c(s+i)=k$;
		\item Let $s=s+k$. Return to step $2$.
	\end{enumerate}
	By $\SPP$, the set $I=\{i\leq k\colon \exists^\infty x\ c(x)=i\}$ exists. One can check that $I\setminus{k}=\{i<k\colon \forall x\exists y\ \theta(i,x,y)\}$.
\end{proof}

The following shows that one direction of the SCT criterion is already provable in $\RCA$.
\begin{proposition}[$\RCAo$]
	Let $\mathcal{G}$ be a finite set of size-change graphs. If every  multipath $M= G_{0},\dots, G_{n},\dots$ has an infinite descent, then  every idempotent $G \in \clG$ has an arc $\arcd{x}{x}$.
\end{proposition}
\begin{proof}
	Let $G$ be idempotent. Then $M = G, G, \dots, G, \dots$ is a multipath. By hypothesis there exists an infinite descent. Since $G$ is idempotent, one can define an infinite sequence $x_0,x_1,x_2,\ldots$ such that 
	$\arcd{x_i}{x_{i+1}}\in G$.  As there are finitely many parameters, by the finite pigeonhole principle, which is provable in $\RCA$, there exist $i<j$ such that $x=x_i=x_j$.   By idempotence of $G$, $\arcd{x}{x} \in G$. 
\end{proof}

\begin{theorem}[$\RCAo$]
	$\STAR$ implies the SCT criterion.
\end{theorem}
\begin{proof}
	We prove the SCT criterion for graphs. Let $\mathcal{G}$ be a finite set of size-change graphs and assume that any idempotent graph in $\clG$ has a strict arc $\arcd{x}{x}$ for some parameter $x$.
	Let
	\[
	\mathcal{M}_\pi= G_0, \dots, G_n, \dots.
	\]
	We aim to prove that $\mathcal{M}_\pi$ has an infinite descent. Define $c:[\NN]^2 \to \clG$ as follows:
	\[
	c(i,j) = G_i; \dots; G_{j-1}.
	\]
	By applying $\STAR_{|\clG|}$ to the coloring $c$, we have:
	\[
	\exists t \exists G \in \clG \forall n \exists m, l ( l > m > n \wedge c(t, m) = c(t, l) = c(m, l) = G ).
	\]
	Then $G$ is idempotent, indeed
	\[
	G;G = c(t, m); c(m, l) = c(t,l) = G.
	\]
	By hypothesis, we have that there exists $\arcd{x}{x}\in G$. Define by primitive recursion a $\STAR$ witness function $g: \NN \to [\NN]^2$ such that for all $n$ 
	\begin{itemize}
		\item $t<g_0(n)<g_1(n)<g_0(n+1)$ and  
		\item $c(t,g_0(n))=c(t,g_1(n))=c(g_0(n),g_1(n))=G$, 
	\end{itemize}
	where $g(n)=(g_0(n),g_1(n))$. 
	
	We claim that there exists an infinite descent starting from $x$ in $G_t$. Since $\arcd{x}{x}\in c(t,g_0(0))$, it is sufficient to show that $\arcd{x}{x}\in c(g_0(n),g_0(n+1))$ for any $n$.
	As $\arcd{x}{x} \in c(t,g_0(n+1))$, there exists $y$ such that $x\to y \in c(t,g_1(n))$ and $y\to x \in c(g_1(n),g_0(n+1))$, and at least one of them is strict. Now $c(t,g_1(n))=c(g_0(n),g_1(n))$, and so $x\to y \in c(g_0(n),g_1(n))$. Therefore we have $\arcd{x}{x}\in c(g_0(n),g_0(n+1))$, as desired. 
\end{proof}

\begin{theorem}[$\RCAo$]\label{reversal}
	The SCT criterion implies $\SPP$.
\end{theorem}

\begin{proof}
	We show that the SCT criterion for graphs implies $\SPP$.
	
	We first prove the thesis for $k=2$. This serves as an illustration of the general case. Note in fact that $\SPP_k$ is provable in $\RCA$ for every standard $k\in\N$.   
	
	Given $c:\NN \to 2$, we want to show that there exists $I\subseteq 2$ such that $i\in I$ iff  $\exists^\infty x\ c(x)=i$.  Let us define $\mcl G$ as follows. The set $\mcl G$ consists of three size-change graphs $G_0,G_1,G_2$ on parameters $z_0,z_1,z_2$. For $i<3$, the graph $G_i$ has only one strict arc $\arcd{z_i}{z_i}$ and non-strict arcs $\arce{z_j}{z_j}$ for $j>i$. Note that every $G\in\clG$ contains a strict arc $\arcd{z}{z}$. Therefore, by the SCT criterion, every multipath of $\clG$ has an infinite descent. Let
	\[
	g(x)=
	\begin{cases}
	0 &\mbox{if }   c(x)=0 \land c(x+1)=0\\
	1 &\mbox{if }  c(x)=1\land c(x+1)=1\\
	2 & \mbox{otherwise}
	\end{cases}
	\]
	Consider the multipath $M=G_{g(0)},G_{g(1)},\ldots$.  Hence there exists an infinite descent in $\mathcal{M}$.  This implies that there exists a parameter $z_i$  that is strictly decreasing infinitely many times, that is $\arcd{z_i}{z_i}\in G_{g(x)}$, viz.\ $g(x)=i$, for infinitely many $x$. If $i<2$, it means that from some point on $c(x)=i$ and so $I=\{i\}$. If $i=2$, then the color changes infinitely many times and so $I=\{0,1\}$.\smallskip

	General case. Let $c\colon\N\to k$ be a given coloring. 
	Let $\mathcal{I}$ be the set of nonempty subset of $k$ and 
	$\sigma_I\in k^{<\N}$ be a fixed enumeration of $I$ for each $I\in\mathcal{I}$. We say that $\chi\colon \mathcal{I}\to k$ is a \emph{choice function} if $\chi(I)\in I$ for all $I\in\mathcal{I}$. We write $\chi_I$ for $\chi(I)$.
	
	Let $\var(\mcl I)$ consist of parameters $z_I$ for every $I\in\mathcal{I}$.
	Given a choice function $\chi$ and a color $i\in k$, we define a size-change graph $G=G_{\chi,i}$ on  $(\var(\mcl I),\var(\mcl I))$ as follows.  Let 
	\[       \mcl A=\mcl A_{\chi,i}=\{I\in \mathcal{I}\colon  \chi_I=\last(\sigma_I)\text{ and } \sigma_I(0)=i\}. \]
	Note that $\mcl A$ is nonempty. For instance $I=\{i\}\in \mcl A$. 
	Let $m\in\N$ be maximum such that there exists $I\in \mcl A$ with $|I|=m$. Define $G$ by letting
	\begin{itemize}
	 \item $\arcd{z_I}{z_I}\in G$ iff $I\in \mcl A$ and $|I|=m$;
	 \item $\arce{z_I}{z_I}\in G$ iff $I\notin \mcl A$ and $|I|\geq m$. 
	\end{itemize} 
	Let $\mathcal{G}=\{G_{\chi,i}\colon  \chi:\mcl{I}\to k\text{ is a choice function and } i\in k\}$. 
	\begin{claim} 
	Every idempotent graph $G\in\clG$ has an arc $\arcd{z_I}{z_I}$ for some $I$. 
	\end{claim}
        \begin{proof} We show that every graph $G\in\clG$ has a strict arc $\arcd{z_I}{z_I}$ for some $I$. Let $G=G_0;G_1;\ldots;G_{l-1}$ with $G_s\in\mathcal{G}$ for all $s<l$. Let $\mcl A_s$ be the $\mcl A_{\chi,i}$ corresponding to $G_s$. Choose $I\in\bigcup_{s<l} \mcl A_s$ of maximum size. We claim that $\arcd{z_I}{z_I}\in G$. Let $p<l$ be such that $I\in \mcl A_p$. By definition,  $\arcd{z_I}{z_I}\in G_p$. 
	By using the maximality of $I$ it is easy to show that for every $s<l$ either $\arcd{z_I}{z_I}\in G_s$ or  $\arce{z_I}{z_I}\in G_s$. \end{proof}
	
	By primitive recursion we now define a function $\chi\colon\N\times\mcl I\to k$ such that $\chi(x,I)\in I$ for all $x\in\N$ and $I\in\mcl I$. We write $\chi_I(x)$ for $\chi(x,I)$. Let $\chi_I(0)=\sigma_I(0)$ for all $I$. Suppose we have defined $\chi_I(x)$. Let
	\[
	\chi_I(x+1)=
	\begin{cases}  \sigma_I(0) & \mbox{if } |I|=1;\\
	c(x) &\mbox{if }   \chi_I(x)=\sigma_I(i)\text{ and } \sigma_I(i^+)=c(x) \mbox{ for some } i<|I|;\\
	\chi_I(x) & \mbox{otherwise}.
	\end{cases}
	\]
	Here $i^+=i+1 \mod |I|$.
	For all $x\in\N$ let $\chi(x)\colon \mcl I\to k$ be the choice function that maps $I$ to $\chi_I(x)$ and
	\[ \mcl A_x=\mcl A_{\chi(x),c(x)}=\{I\in\mcl I\colon \chi_I(x)=\last(\sigma_I) \text{ and }\sigma_I(0)=c(x)\}.\]
	Define a multipath $M=G_0,G_1,\ldots$ by $G_x=G_{\chi(x),c(x)}$ for every $x\in\N$. 
	\begin{claim}
		For every $I\in\mathcal{I}$ the following are equivalent:
		\begin{enumerate}
			\item[(i)] Every color in $I$ appears infinitely often;
			\item[(ii)] $I\in\mcl A_x$ for infinitely many $x$.
		\end{enumerate}
	\end{claim}
	\begin{proof} 
	        The claim clearly holds for $|I|=1$. Suppose $|I|>1$. 
	        
		(i) implies (ii). Let $s$ be given. We aim to show that for some $x\geq s$ we have that $\chi_I(x)=\last(\sigma_I)$ and $\sigma_I(0)=c(x)$. First we claim that for every $i\in I$ there exists $x\geq s$ such that $\chi_I(x)=i$. Let $i_0<|I|$ be such that $\chi_I(s)=\sigma_I(i_0)$. By $\Sigma^0_1$-induction one can show that for every $j<|I|$ there exists $x\geq s$ such that $\chi_I(x)=\sigma_I(i_j)$, where $i_j=i_0+j\mod |I|$. The thesis follows since every $i\in I$ is of the form $\sigma_I(i_j)$ for some $j<|I|$. This proves that for every $i\in I$ there exists $x\geq s$ such that $\chi_I(x)=i$. Now let $y\geq s$ be such that $\chi_I(y)=\last(\sigma_I)$. Since every color in $I$ appears infinitely often, there exists $x\geq y$ such that $c(x)=\sigma_I(0)$. Let $x\geq y$ be least with this property. Then by $\Sigma^0_0$-induction one can show that $\chi_I(x)=\last(\sigma_I)$. 
		
		It remains to show (ii) implies (i).  The argument is similar to that of (i) implies (ii). Let $i_0+1=|I|$.  By hypothesis we have that $\chi_I(x)=\sigma_I(i_0)$ for infinitely many $x$. Given $s\in\N$, one can show by $\Sigma^0_1$-induction that for every $j<|I|$ there exists $x\geq s$ such that $c(x)=\sigma_I(i_j)$, where $i_j$ is defined as above.  
	\end{proof}
	
	 By the SCT criterion for graphs, we have an infinite descent in $M$ for some parameter $z_I$ starting at some point $t$.  We aim to show that 
	 \[ I=\{ i<k\colon \exists^\infty x\ c(x)=i\}. \] Now, there exist infinitely many $x$ such that $\arcd{z_I}{z_I}\in G_x$, and in particular $I\in\mcl A_x$. By the claim  we have that every color in $I$ appears infinitely many times. We claim that every color in $k\setminus I$ appears finitely many times.
	 Suppose not and let $J\supset I$ be such that every color in $J$ appears infinitely often.  By the claim there exists $x>t$ such that $J\in\mcl A_x$. By definition, in $G_x$ there is no arc from $z_I$ to $z_I$, a contradiction. 
\end{proof}

Therefore we can conclude that $ \STAR \geq \text{SCT criterion} \geq \is2$. Actually we can prove that they are all equivalent. 

\begin{theorem}\label{equiv}
	Over $\RCAo$ the following are equivalent:
	\begin{enumerate}
		\item $\is2$
		\item $\STAR$
		\item SCT criterion
		\item SCT criterion for graphs
	\end{enumerate}
\end{theorem}
\begin{proof}
	We need only to show that $\is2$ implies $\STAR$. 
	As shown in \cite{SlamanKeita} $\RT^2$ is $\Pi^1_1$-conservative over $\bs3$, the bounding principle for $\Sigma^0_3$-formulas. 
	So, since $\RT^2$ trivially implies $\STAR$ (which is a $\Pi^1_1$-statement), then also $\bs3$ does.  It is known that $\bs3$ is $\widetilde{\Pi}^0_4$-conservative over $\is2$, where a statement is $\widetilde{\Pi}^0_4$ if it is of the form $\forall X \varphi(X)$ and $\varphi(X) \in \Pi^0_4$. This follows as a particular case from the analogue result in first order arithmetic that $\mathsf{B}\Sigma_{n+1}$ is $\Pi_{n+2}$-conservative over $\mathsf{I}\Sigma_n$ for all $n\geq 0$ (see \cite[Chapter IV, Section 1(f)]{Hajek}). Finally, one can check that $\STAR$ is $\widetilde{\Pi}^0_4$, hence the thesis. 
\end{proof}

\begin{remark}
	One can directly show that the P\'{e}ter-Ackermann function  is SCT in both senses. Indeed let $G_{0,1}, G_2$ be the size change graphs of the  P\'{e}ter-Ackermann function as in Example \ref{exampleAck}.
	Let $\mcl M=G'_0, \dots, G'_n, \dots$ be an infinite multipath. We have
	\[ \forall n\ \exists m \geq n\ G'_n=G_{0,1} \vee \exists n\ \forall m\geq n\  G'_n=G_2.       \]
    In the first case we have an infinite descent for $x$ starting in $G'_0$. 
    The second case yields an infinite descent for $y$ starting in some $G_n$, since all graphs in the multipath from $n$ on are $G_2$. Note that this proof is in classical logic, since it requires the Law of Excluded Middle. 
	\end{remark}


In general, if $\mcl G$ has size $k$ for some standard $k\in\omega$, then $\RCA$ proves the SCT criterion for $\mcl G$. This follows from the following:

\begin{proposition}
	For any standard $k \in \omega$, 
	\[
	\RCAo \vdash \STAR_k.
	\]
\end{proposition}
\begin{proof}
	Note that $\RCAo\vdash \RT^1_k$ for all standard $k\in\NN$. We prove $\STAR_k$ by (external) induction on $k$.
	
	Given a coloring $c : [\NN]^{2} \to k$, let $c_0 : \NN \to k$ such that $c_0(x)= c(0,x)$ and let $X$ be the infinite homogeneous set given by $\RT^1_k$.  Let $ \bp{x_n : n \in \NN}$ be the increasing enumeration of $X$.  
	Suppose $i = c(0,x_0)$. By the law of excluded middle, we have:
	\[
	\forall n \exists m, l ( l > m > n \wedge c(x_m, x_l) = i) \ \vee \ \exists n \forall m, l ( l > m > n \implies c(x_m, x_l) \neq i).
	\]
	In the first case we are done. In the second case let $Y=\bp{ x \in X : x > x_n}$. Then $Y$ is an infinite homogeneous set in $(k-1)$-many colors. By the induction hypothesis (on $d : [\NN]^{2} \to k-1$ such that $d(a,b) = c(x_{n+a}, x_{n+d})$) we are done again.
\end{proof}

\section{Conclusion and further works}

In this paper we addressed the study of size-change analysis in the context of Reverse Mathematics. We determined the exact strength of the SCT criterion by proving that it is equivalent to a weak version of Ramsey's theorem for pairs, which turns out to be equivalent to $\Sigma^0_2$-induction over $\RCA$. In particular the proof of the SCT criterion does not require full Ramsey's theorem for pairs.  

One of the motivations for studying size-change termination in the framework of Reverse Mathematics is that the P\'{e}ter-Ackermann function is size-change-terminating. Actually, this can be proved in $\RCA$, whereas it is well known that the totality of the P\'{e}ter-Ackermann function is not provable in $\RCA$. This arises the question of what is needed in order to show the \emph{soundness} of size-change termination (SCT soundness), that is the statement that every SCT program terminates. 

The classical proof is based on the fact that ``if a program does not terminate then there exists an infinite state transition sequence''. This statement seems to require  K\"{o}nig's lemma, which is equivalent to Arithmetical Comprehension Axiom ($\ACA$) over the base system $\RCA$. Roughly, $\ACA$ asserts the existence of the jump of every set of natural numbers. 

We suspect that a direct proof of the SCT soundness does not require any comprehension (set existence) axiom.  In fact, it is known that SCT programs compute exactly the \emph{multiply recursive} functions \cite{BenAmram}. On the other hand, the 
class of multiply recursive functions coincides with the class $\mcl M=\bigcup_{\alpha<\omega^\omega}\mcl{F}_{\alpha}$, where $(\mcl{F}_\alpha)_\alpha$ is the fast growing hierarchy \cite{Lob}. Since well-foundedness of $\omega^{\omega^\omega}$ implies the totality of every function in $\mcl M$, we thus conjecture that SCT soundness is provable in $\RCA$ plus well-foundedness of $\omega^{\omega^\omega}$.

\bibliographystyle{plain}
\bibliography{biblio} 
\end{document}